\documentclass{amsart}

\usepackage{geometry}
\usepackage{amsmath}
\usepackage{amssymb}
\usepackage{amsthm}
\usepackage{mathrsfs}
\usepackage{enumitem}
\usepackage{setspace}
\usepackage{tikz-cd}
\usepackage{verbatim}
\usepackage{mathdots}
\usepackage{bbm}
\usepackage{hyperref}
\usepackage{cite}

\newcommand{\Z}{\mathbb{Z}}

\newcommand{\E}[1]{\mathbb{E} \left[ #1 \right]}
\renewcommand{\P}{\mathbb{P}}

\DeclareMathOperator{\aut}{Aut}

\DeclareMathOperator{\dist}{dist}

\DeclareMathOperator{\prob}{Prob}
\DeclareMathOperator{\pat}{pat}
\DeclareMathOperator{\repp}{rp}
\DeclareMathOperator{\cov}{cov}

\renewcommand{\a}{\alpha}
\renewcommand{\b}{\beta}

\newcommand{\bea}{\begin{enumerate}[label=(\alph*)]}
\newcommand{\ee}{\end{enumerate}}
\newcommand{\bal}{\begin{align*}}
\newcommand{\beq}{\begin{equation}}
\newcommand{\eeq}{\end{equation}}

\newtheorem{thm}{Theorem}[section]
\newtheorem{cor}[thm]{Corollary}
\newtheorem{lem}[thm]{Lemma}
\newtheorem{lemma}[thm]{Lemma}
\newtheorem{prop}[thm]{Proposition}

\theoremstyle{definition}
\newtheorem{defn}[thm]{Definition}

\renewcommand{\bar}[1]{\overline{#1}}

\newcommand{\wt}{\widetilde}

\newcommand{\B}{B_{\operatorname{Ham}}}
\newcommand{\sC}{\mathscr{C}}

\title{Zero entropy actions of amenable groups are not dominant}
\author{Adam Lott}
\address{Department of Mathematics, University of California, Los Angeles, Los Angeles, CA 90095}
\email{adamlott99@math.ucla.edu}
\date{\today}

\begin{document}

\begin{abstract}
    A probability measure preserving action of a discrete amenable group $G$ is said to be \emph{dominant} if it is isomorphic to a generic extension of itself.  In \cite{austin2021dominant}, it was shown that for $G = \Z$, an action is dominant if and only if it has positive entropy and that for any $G$, positive entropy implies dominance.  In this paper we show that the converse also holds for any $G$, i.e. that zero entropy implies non-dominance.  
\end{abstract}

\maketitle{}


\section{Introduction}
\subsection{Definitions and results}  
Let $(X,\mathcal{B}, \mu)$ be a standard Lebesgue space and let $T$ be a free, ergodic, $\mu$-preserving action of a discrete amenable group $G$ on $X$.  It is natural to ask what properties of $T$ are preserved by a generic extension $(\bar{X}, \bar{\mu}, \bar{T})$ (a precise definition of ``generic extension'' is discussed in section \ref{sec: Cocycles}).  For example, it was shown in \cite{glasner2021generic} that if $T$ is a nontrivial Bernoulli shift then a generic $\bar{T}$ is also Bernoulli and that a generic $\bar{T}$ has the same entropy as $T$.  A system $(X, \mu, T)$ is said to be {\bf dominant} if it is isomorphic to a generic extension $(\bar{X}, \bar{\mu}, \bar{T})$.  So, for example, the aforementioned results from \cite{glasner2021generic} together with Ornstein's famous isomorphism theorem \cite{ornstein1970bernoulli} imply that all nontrivial Bernoulli shifts are dominant.  More generally, it has been shown in \cite{austin2021dominant} that
\begin{enumerate}
    \item if $G = \Z$, then $(X,\mu,T)$ is dominant if and only if it has positive Kolmogorov-Sinai entropy, and
    
    \item for any $G$, if $(X, \mu, T)$ has positive entropy, then it is dominant.
\end{enumerate}
In this paper, we complete the picture by proving the following result.

\begin{thm} \label{thm: MainResultIntroVersion}
Let $G$ be any discrete amenable group, and let $(X, \mu, T)$ be any free ergodic action with zero entropy.  Then $(X, \mu, T)$ is not dominant.  
\end{thm}

The proof of result (2) above is based on the theory of ``slow entropy'' developed by Katok and Thouvenot in \cite{katok1997slow}, and our proof of Theorem \ref{thm: MainResultIntroVersion} uses the same ideas.

\subsection{Outline}
In section \ref{sec: SlowEntropy}, we introduce the relevant ideas from slow entropy.  In section \ref{sec: Cocycles}, we describe a precise definition of ``generic extension'' and begin the proof of Theorem \ref{thm: MainResultIntroVersion}.  Finally, in section \ref{sec: ConstructionSection}, we prove the proposition that is the technical heart of Theorem \ref{thm: MainResultIntroVersion}.

\subsection{Acknowledgements} 
I am grateful to Tim Austin for originally suggesting this project and for endless guidance.  I also thank Benjy Weiss for pointing out an incorrect reference in an earlier version.

This project was partially supported by NSF grant DMS-1855694.

\section{Slow entropy} \label{sec: SlowEntropy}

Fix a F\o lner sequence $(F_n)$ for $G$.  For $g \in G$, write $T^g x$ for the action of $g$ on the point $x \in X$, and for a subset $F \subseteq G$, write $T^F x = \{T^f x : f \in F \}$.  If $Q = \{Q_1, \dots, Q_k\}$ is a partition of $X$, then for $x \in X$ denote by $Q(x)$ the index of the cell of $Q$ containing $x$.  Sometimes we will use the same notation to mean the cell itself; which meaning is intended will be clear from the context.
Given a finite subset $F \subseteq G$, the {\bf $\mathbf{(Q,F)}$-name} of $x$ for the action $T$ is the tuple $Q_{T,F}(x) := (Q(T^f x))_{f \in F} \in \{1,2,\dots,k\}^F$.  
Similarly, we also define the partition $Q_{T,F} := \bigvee_{f \in F} T^{f^{-1}} Q$, and in some contexts we will use the same notation $Q_{T,F}(x)$ to refer to the cell of $Q_{T,F}$ containing $x$.

For a finite subset $F \subseteq G$ and any finite alphabet $\Lambda$, the symbolic space $\Lambda^F$ is equipped with the normalized Hamming distance $d_F(w,w') = \frac{1}{|F|} \sum_{f \in F} 1_{w(f) \neq w'(f)}$.

\begin{defn}
Given a partition $Q = \{Q_1, \dots, Q_k\}$, a finite set $F \subseteq G$, and $\epsilon > 0$, define 
\[\
\B(Q,T,F,x,\epsilon) \ := \ \{y \in X : d_F(Q_{T,F}(y), Q_{T,F}(x)) < \epsilon \}.
\]
We refer to this set as the ``$(Q,T,F)$-Hamming ball of radius $\epsilon$ centered at $x$''.  Formally, it is the preimage under the map $Q_{T,F}$ of the ball of radius $\epsilon$ centered at $Q_{T,F}(x)$ in the metric space $([k]^F, d_F)$.
\end{defn}

\begin{defn}
The {\bf Hamming covering number} of $\mu$ is defined to be the minimum number of $(Q,T,F)$-Hamming balls of radius $\epsilon$ required to cover a subset of $X$ of $\mu$-measure at least $1-\epsilon$, and is denoted by 
$\cov(Q,T,F,\mu,\epsilon)$.  
\end{defn}

\begin{lemma} \label{lem: TransformHammingCovNumsUnderIso}
Let $\varphi: (X,T,\mu) \to (Y, S, \nu)$ be an isomorphism.  Also let $Q$ be a finite partition of $X$, $F$ a finite subset of $G$, and $\epsilon > 0$.  Then $\cov(Q,T,F,\mu,\epsilon) = \cov(\varphi Q, S, F, \nu, \epsilon)$.
\end{lemma}

\begin{proof}
It is immediate from the definition of isomorphism that for $\mu$-a.e. 
$x,x' \in X,$
\[
d_F\left( Q_{T,F}(x), Q_{T,F}(x') \right) \ = \ d_F\left( (\varphi Q)_{S,F}(\varphi x), (\varphi Q)_{S,F}(\varphi x') \right).
\]
Therefore it follows that $\varphi \left( \B(Q,T,F,x,\epsilon) \right) = \B(\varphi Q, S, F, \varphi x, \epsilon)$ for $\mu$-a.e. $x$.  So any collection of $(Q,T,F)$-Hamming balls in $X$ covering a set of $\mu$-measure $1-\epsilon$ is directly mapped by $\varphi$ to a collection of $(\varphi Q, S, F)$-Hamming balls in $Y$ covering a set of $\nu$-measure $1-\epsilon$.  Therefore $\cov(Q,T,F,\mu,\epsilon) \geq \cov(\varphi Q, S, F, \nu, \epsilon)$.  The reverse inequality holds by doing the same argument with $\varphi^{-1}$ in place of $\varphi$. 
\end{proof}

The goal of the rest of this section is to show that for a given action $(X,T,\mu)$, the sequence of covering numbers $\cov(Q, T, F_n, \mu, \epsilon)$ grows at a rate that is bounded uniformly for any choice of partition $Q$.  A key ingredient is an analogue of the classical Shannon-McMillan theorem for actions of amenable groups \cite[Theorem 4.4.2]{ollagnier1985book}.

\begin{thm} \label{thm: ShannonMcMillan}
Let $G$ be a countable amenable group, and let $(F_n)$ be any F\o lner sequence for $G$.  Let $(X,T,\mu)$ be an ergodic action of $G$, and let $Q$ be any finite partition of $X$.  Then
\[
\frac{-1}{|F_n|} \log \mu \left( Q_{T, F_n}(x) \right) \ \xrightarrow{L^1(\mu)} \ h(\mu, T, Q) \qquad \text{as $n \to \infty$}, 
\]
where $h$ denotes the entropy.  In particular, for any fixed $\gamma > 0$,
\[
\mu \left\{ x : \exp((-h-\gamma) |F_n|) < \mu(Q_{T,F_n}(x)) < \exp((-h+\gamma)|F_n|) \right\} \ \to \ 1 \qquad \text{as $n \to \infty$}.
\]
\end{thm}

\begin{lem} \label{lem: NumberOfCellsToCover}
For any partition $P$, any F\o lner sequence $(F_n)$, and any $\epsilon
> 0$, let $\ell(T, P, \mu, n)$ be the minimum number of $P_{T, F_n}$-cells required to cover a subset of $X$ of measure $> 1-
\epsilon$.  Then
\[
\limsup_{n \to \infty} \frac{1}{|F_n|} \log \ell(T, P, \mu, n) \ \leq \ h(\mu, T, P).
\]
\end{lem}
\begin{proof}
Let $h = h(\mu, T, P)$.  Let $\gamma > 0$.  By Theorem \ref{thm: ShannonMcMillan}, for $n$ sufficiently large depending on $\gamma$, we have 
\[
\mu \left\{ x \in X : \mu(P_{T, F_n}(x)) \geq \exp((-h-\gamma)|F_n|) \right\} \ > \ 1-
\epsilon.
\]
Let $X'$ denote the set in the above equation.  Let $\mathcal{G}$ be the family of cells of the partition $P_{T, F_n}$ that meet $X'$.  Then clearly $\mu \left( \bigcup \mathcal{G} \right) > 1-
\epsilon$ and $|\mathcal{G}| < \exp((h+\gamma) |F_n|)$.  Therefore
\[
\limsup_{n \to \infty} \frac{1}{|F_n|} \log \ell(n) \ \leq \ h + \gamma,
\]
and this holds for arbitrary $\gamma$, so we are done.
\end{proof}



At this point, fix for all time $\epsilon 
= 1/100$.  We can also now omit $\epsilon$ 
from all of the notations defined above, because 
it will never change.  Also assume from now on that the system $(X, T, \mu)$ has zero entropy.

\begin{lem} \label{lem: ThickenFolner}
If $(F_n)$ is a F\o lner sequence for $G$ and $A$ is any finite subset of $G$, then $(AF_n)$ is also a F\o lner sequence for $G$.
\end{lem}
\begin{proof}
First, because $A$ is finite and $(F_n)$ is F\o lner we have
\[
\lim_{n \to \infty} \frac{|AF_n|}{|F_n|} \ = \ 1.
\]
Now fix any $g \in G$ and observe that
\[
\frac{|gAF_n \,\triangle\, AF_n|}{|AF_n|} \ \leq \ \frac{|gAF_n \,\triangle\, F_n| + |F_n \,\triangle\, AF_n|}{|F_n|} \cdot
\frac{|F_n|}{|AF_n|} \ \to \ 0 \qquad \text{as $n \to \infty$},
\]
which shows that $(AF_n)$ is a F\o lner sequence.
\end{proof}

\begin{lem} \label{lem: Diagonalize}
Let $b(m,n) \geq 0$ be real numbers satisfying
\begin{itemize}
    \item $\lim_{n \to \infty} b(m,n) = 0$ for each fixed $m$, and
    \item $b(m+1, n) \geq b(m,n)$ for all $m,n$.
\end{itemize}
Then there exists a sequence $(a_n)$ such that $a_n \to 0$ and for each fixed $m$, $b(m,n) \leq a_n$ for $n$ sufficiently large (depending on $m$).
\end{lem}
\begin{proof}
For each $m$, let $N_m$ be such that $b(m,n) < 1/m$ for all $n > N_m$.  Without loss of generality, we may assume that $N_m < N_{m+1}$.  Then we define the sequence $(a_n)$ by $a_n = b(1,n)$ for $n \leq N_2$ and $a_n = b(m,n)$ for $N_m < n \leq N_{m+1}$.  We have $a_n \to 0$ because $a_n < 1/m$ for all $n > N_m$.  Finally, the fact that $b(m+1, n) \geq b(m,n)$ implies that for every fixed $m$, $a_n \geq b(m,n)$ as soon as $n > N_m$.
\end{proof}

\begin{prop} \label{prop: UniformGrowthRate}
There is a sequence $(a_n)$ such that
\begin{enumerate}
    \item $\limsup_{n \to \infty} \frac{1}{|F_n|} \log a_n = 0$, and
    \item for any finite partition $Q$, there exists an $N$ such that 
    $\cov(Q,T,F_n, \mu) \leq a_n$ for all $n > N$.
    \end{enumerate}
\end{prop}

\begin{proof}
Because $T$ has zero entropy, there exists a finite generating partition for $T$ (see for example \cite[Corollary 1.2]{seward2019krieger} or \cite[Theorem 2$'$]{rosenthal1988generators}).  Fix such a partition $P$ and let $Q = \{Q_1, \dots, Q_r\}$ be any given partition.  Because $P$ is generating, there is an integer $m$ and another partition $Q' = \{Q'_1, \dots, Q'_r \}$ such that $Q'$ is refined by $P_{T, F_m} = \bigvee_{g \in F_m} (T^g)^{-1}P$ and 
\[
\mu \{ x : Q(x) \neq Q'(x) \} \ < \ \frac{\epsilon}{4}.
\]

By the mean ergodic theorem, we can write
\[
d_{F_n}(Q_{T, F_n}(x), Q'_{T, F_n}(x)) \ = \ \frac{1}{|F_n|} \sum_{f \in F_n} 1_{\{ y : Q(y) \neq Q'(y) \}}(T^f x) \ \xrightarrow{L^1(\mu)} \ \mu \{ y : Q(y) \neq Q'(y) \} \ < \ \frac{\epsilon}{4},
\]
so in particular, for $n$ sufficiently large, we have
\[
\mu \{ x : d_{F_n}(Q_{T, F_n}(x), Q'_{T, F_n}(x)) < \epsilon/2 \} \ > \ 1 - \frac{
\epsilon}{4}.
\]
Let $Y$ denote the set $\{ x : d_{F_n}(Q_{T, F_n}(x), Q'_{T, F_n}(x)) < \epsilon/2 \}$.  

Recall that $Q'$ is refined by $P_{T, F_m}$, so $Q'_{T, F_n}$ is refined by $(P_{T, F_m})_{T, F_n} = P_{T, F_m F_n}$.  Let $\ell = \ell(m,n)$ be the minimum number of $P_{T, F_m F_n}$ cells required to cover a set of $\mu$-measure at least $1-\epsilon/4$, and let $C_1, \dots, C_{\ell}$ be such a collection of cells satisfying $\mu \left( \bigcup_i C_i \right) \geq 1 - 
\epsilon/4$.  If any of the $C_i$ do not meet the set $Y$, then drop them from the list.  Because $\mu(Y) > 1-
\epsilon/4$ we can still assume after dropping that $\mu \left( \bigcup_i C_i \right) > 1 - 
\epsilon/2$.  Choose a set of representatives $y_1, \dots y_\ell$ with each $y_i \in C_i \cap Y$.  

Now we claim that $Y \cap \bigcup_i C_i \subseteq \bigcup_{i=1}^{\ell} \B(Q, T, F_n, y_i, \epsilon)$.  To see this, let $x \in Y \cap \bigcup_i C_i$.  Then there is one index $j$ such that $x$ and $y_j$ are in the same cell of $P_{T, F_mF_n}$.  We can then estimate
\begin{align*}
d_{F_n}\left( Q_{T, F_n}(x), Q_{T, F_n}(y_j) \right) \ &\leq \ d_{F_n}\left( Q_{T, F_n}(x), Q'_{T, F_n}(x) \right) + d_{F_n}\left( Q'_{T, F_n}(x), Q'_{T, F_n}(y_j) \right)\\
&\quad + d_{F_n}\left( Q'_{T, F_n}(y_j), Q_{T, F_n}(y_j) \right) \\
&< \ \frac{\epsilon}{2} + 0 + \frac{\epsilon}{2} \ = \ \epsilon.
\end{align*}
The bounds for the first and third terms come from the fact that $x, y_j \in Y$.  The second term is $0$ because $Q'_{F_n}$ is refined by $P_{T, F_m F_n}$ and $y_j$ was chosen so that $x$ and $y_j$ are in the same $P_{T, F_m F_n}$-cell.  Therefore, 
$\cov(Q, T, F_n, \mu) \leq \ell(m,n)$.  So, the proof is complete once we find a fixed sequence $(a_n)$ that is subexponential in $|F_n|$ and eventually dominates $\ell(m,n)$ for each fixed $m$.

Because $T$ has zero entropy, Lemmas \ref{lem: NumberOfCellsToCover} and \ref{lem: ThickenFolner} imply that
\[
\limsup_{n \to \infty} \frac{1}{|F_n|} \log \ell(m,n) \ = \ \limsup_{n \to \infty} \frac{|F_m F_n|}{|F_n|} \cdot \frac{1}{|F_m F_n|} \log \ell(m,n) \ = \ 0 \qquad \text{for each fixed $m$}.
\]
Note also that because $P_{T, F_{m+1}F_n}$ refines $P_{T, F_m F_n}$, we have $\ell(m+1, n) \geq \ell(m,n)$ for all $m,n$.  Therefore, we can apply Lemma \ref{lem: Diagonalize} to the numbers $b(m,n) = |F_n|^{-1} \log \ell(m,n)$ to produce a sequence $(a_n')$ satisfying $a_n' \to 0$ and $a_n' \geq b(m,n)$ eventually for each fixed $m$.  Then $a_n := \exp(|F_n| a_n')$ is the desired sequence.
\end{proof}

\section{Cocycles and extensions} \label{sec: Cocycles}

Let $I$ be the unit interval $[0,1]$, and let $m$ be Lebesgue measure on $I$.  Denote by $\aut(I,m)$ the group of invertible $m$-preserving transformations of $I$.  A {\bf cocycle} on $X$ is a family of measurable maps $\a_g : X \to \aut(I,m)$ indexed by $g \in G$ that satisfies the {\bf cocycle condition}: for every $g,h \in G$ and $\mu$-a.e. $x$, $\a_{hg}(x) = \a_h(T^g x) \circ \a_g(x)$.  A cocycle can equivalently be thought of as a measurable map $\a: R \to \aut(I,m)$, where $R \subseteq X \times X$ is the orbit equivalence relation induced by $T$ (i.e. $(x,y) \in R$ if and only if $y = T^g x$ for some $g \in G$).  With this perspective, the cocycle condition takes the form $\a(x,z) = \a(y,z) \circ \a(x,y)$.  A cocycle $\a$ induces the skew product action $T_\a$ of $G$ on the larger space $X \times I$ defined by 
\[
T_\a^g(x,t) \ := \ (T^g x, \a_g(x) (t)).
\]
This action preserves the measure $\mu \times m$ and factors onto the original action $(X,T,\mu)$.

By a classical theorem of Rokhlin (see for example \cite[Theorem 3.18]{glasner2003ergodic}), any infinite-to-one extension of $(X,\mu,T)$ is isomorphic to a skew product action for some cocycle.  Therefore, by topologizing the space of all cocycles on $X$ we can capture the notion of a ``generic'' extension -- a property is said to hold for a generic extension if it holds for a dense $G_\delta$ set of cocycles.  Denote the space of all cocycles on $X$ by $\operatorname{Co}(X)$.  Topologizing $\operatorname{Co}(X)$ is done in a few stages.
\begin{enumerate}
    \item Let $\mathcal{B}(I)$ be the Borel sets in $I$, and let $(E_n)$ be a sequence in $\mathcal{B}(I)$ that is dense in the $\mu ( \cdot \,\triangle\, \cdot)$ metric.  For example, $(E_n)$ could be an enumeration of the family of all finite unions of intervals with rational endpoints.
    
    \item $\aut(I,m)$ is completely metrizable via the metric 
    \[
    d_A(\phi, \psi) \ = \ \frac12 \sum_{n \geq 1} 2^{-n} [m (\phi E_n \,\triangle\, \psi E_n) + m(\phi^{-1} E_n \,\triangle\, \psi^{-1} E_n)].
    \]
    Notice that with this metric, $\aut(I,m)$ has diameter at most $1$.  See for example \cite[Section 1.1]{kechris2010global}
        
    \item If $\a_0, \b_0$ are maps $X \to \aut(I,m)$, then define $\dist(\a_0, \b_0) = \int d_A(\a_0(x), \b_0(x)) \,d\mu(x)$.
        
    \item The metric defined in the previous step induces a topology on $\aut(I,m)^X$.  Therefore, because $\operatorname{Co}(X)$ is just a certain (closed) subset of $(\aut(I,m)^X)^G$, it just inherits the product topology. 
\end{enumerate}
    
To summarize, if $\a$ is a cocycle, then a basic open neighborhood $\a$ is specified by two parameters: a finite subset $F \subseteq G$ and $\eta > 0$.  The $(F,\eta)$-neighborhood of $\a$ is $\{ \beta \in \operatorname{Co}(X): \dist(\a_g, \b_g) < \eta \text{ for all $g \in F$} \}$.  In practice, we will always arrange things so that $\a_g(x) = \b_g(x)$ for all $g \in F$ on a set of $x$ of measure $\geq 1-\eta$, which is sufficient to guarantee that $\b$ is in the $(F,\eta)$-neighborhood of $\a$.

Let $\bar{Q}$ be the partition $\{X \times [0,1/2], X \times (1/2, 1] \}$ of $X \times I$.  We derive Theorem \ref{thm: MainResultIntroVersion} from the following result about covering numbers of extensions, which is the main technical result of the paper.

\begin{thm} \label{thm: TechnicalMainResult}
For any sequence $(a_n)$ satisfying $\limsup_{n \to \infty} \frac{1}{|F_n|} \log a_n = 0$, there is a dense $G_\delta$ set $\mathcal{U} \subseteq \operatorname{Co}(X)$ such that for any $\a \in \mathcal{U}$,
$\cov(\bar{Q}, T_\a, F_n, \mu\times m) > 2a_n$ for infinitely many $n$. 
\end{thm}


\begin{proof}
[Proof that Theorem \ref{thm: TechnicalMainResult} implies Theorem \ref{thm: MainResultIntroVersion}]
Choose a sequence $(a_n)$ as in Proposition \ref{prop: UniformGrowthRate} such that for any partition $Q$, 
$\cov(Q, T, F_n, \mu) \leq a_n$ for sufficiently large $n$.  Let $\mathcal{U}$ be the dense $G_\delta$ set of cocycles associated to $(a_n)$ as guaranteed by Theorem \ref{thm: TechnicalMainResult}, and let $\a \in \mathcal{U}$, so we know that $\cov(\bar{Q}, T_\a, F_n, \mu \times m) \geq 2a_n$ for infinitely many $n$. 
Now if $\varphi: (X \times I, T_\a, \mu \times m) \to (X, T, \mu)$ were an isomorphism, then by Lemma \ref{lem: TransformHammingCovNumsUnderIso}, $\varphi \bar{Q}$ would be a partition of $X$ satisfying 
$\cov(\varphi \bar{Q}, T, F_n, \mu) = \cov(\bar{Q}, T_\a, F_n, \mu \times m) > 2a_n$ for infinitely many $n$, contradicting the conclusion of Proposition \ref{prop: UniformGrowthRate}.  Therefore, we have produced a dense $G_\delta$ set of cocycles $\a$ such that $T_\a \not\simeq T$, which implies Theorem \ref{thm: MainResultIntroVersion}.
\end{proof}

To prove Theorem \ref{thm: TechnicalMainResult}, we need to show roughly that $\{ \a \in \operatorname{Co}(X) : \cov(\bar{Q}, T_\a, F_n, \mu\times m) \text{ is large} \}$ is both open and dense.  We will address the open part here and leave the density part until the next section.   Let $\pi$ be the partition $\{[0,1/2), [1/2, 1]\}$ of $I$. 

\begin{lem} \label{lem: ConvergingCocycles}
If $\b^{(n)}$ is a sequence of cocycles converging to $\a$, then for any finite $F \subseteq G$, we have
\[
(\mu \times m) \left\{ (x,t) : \bar{Q}_{T_{\b^{(n)}}, F}(x,t) \ = \ \bar{Q}_{T_{\a}, F}(x,t) \right\} \ \to \ 1 \qquad \text{as $n \to \infty$}.
\]
\end{lem}
\begin{proof}
For the names $\bar{Q}_{T_{\b^{(n)}}, F}(x,t)$ and $\bar{Q}_{T_{\a}, F}(x,t)$ to be the same means that for every $g \in F$,
\[
\bar{Q}\left( T^g x, \b^{(n)}_g(x) t \right) \ = \ \bar{Q} \left( T^g x, \a_g(x) t \right),
\]
which is equivalent to
\beq \label{eq: NameAgreementCondition}
\pi\left( \b^{(n)}_g(x) t \right) \ = \ \pi \left( \a_g(x) t \right).
\eeq
The idea is the following.  For fixed $g$ and $x$, if $\a_g(x)$ and $\b^{(n)}_g(x)$ are close in $d_A$, then \eqref{eq: NameAgreementCondition} fails for only a small measure set of $t$.  And if $\b^{(n)}$ is very close to $\a$ in the cocycle topology, then $\b^{(n)}_g(x)$ and $\a_g(x)$ are close for all $g \in F$ and most $x \in X$.  Then, by Fubini's theorem, we will get that the measure of the set of $(x,t)$ failing \eqref{eq: NameAgreementCondition} is small.

Here are the details.  Fix $\rho > 0$; we will show that the measure of the desired set is at least $1-\rho$ for $n$ sufficiently large.  First, let $\sigma$ be so small that for any $\phi, \psi \in \aut(I,m)$, 
\[
d_A(\phi, \psi) < \sigma \quad \text{implies} \quad m \left\{ t : \pi(\phi t) = \pi(\psi t) \right\} > 1-\rho/2.
\]
This is possible because 
\[
\{ t : \pi(\phi t) \neq \pi(\psi t) \} \ \subseteq \ (\phi^{-1}[0,1/2) \,\triangle\, \psi^{-1}[0,1/2)) \cup (\phi^{-1}[1/2,1] \,\triangle\, \psi^{-1}[1/2,1]).
\]
Then, from the definition of the cocycle topology, we have
\[
\mu \left\{ x \in X : d_A \left( \b^{(n)}_g(x), \a_g(x) \right) < \sigma \quad \text{for all $g \in F$} \right\} \ \to \ 1 \qquad \text{as $n \to \infty$}.
\]
Let $n$ be large enough so that the above is larger than $1-\rho/2$.  Then, by Fubini's theorem, we have
\begin{align*}
    &(\mu \times m) \left\{ (x,t) : \bar{Q}_{T_{\b^{(n)}}, F}(x,t) \ = \ \bar{Q}_{T_{\a}, F}(x,t) \right\} \\
    = \ &\int m \left\{ t : \pi\left( \b^{(n)}_g(x) t \right) = \pi \left( \a_g(x) t \right) \quad \text{for all $g \in F$} \right\} \,d\mu(x).
\end{align*}
We have arranged things so that the integrand above is $>1-\rho/2$ on a set of $x$ of $\mu$-measure $> 1-\rho/2$, so the integral is at least $(1-\rho/2)(1-\rho/2) > 1-\rho$ as desired.
\end{proof}

\begin{lem} \label{lem: OpenSetOfCocycles}
For any finite $F \subseteq G$ and any $L > 0$, $\{ \a \in \operatorname{Co}(X) : 
\cov(\bar{Q}, T_\a, F, \mu\times m) > L \}$
is open in $\operatorname{Co}(X)$.
\end{lem}
\begin{proof}
Suppose $\b^{(n)}$ is a sequence of cocycles converging to $\a$ and satisfying 
$\cov(\bar{Q}, T_{\b^{(n)}}, F, \mu \times m) \leq L$ for all $n$.  We will show that 
$\cov(\bar{Q}, T_\a, F, \mu \times m) \leq L$ as well.  The covering number 
$\cov(\bar{Q}, T_{\b^{(n)}}, F, \mu \times m)$ is a quantity which really depends only on the measure $\left( \bar{Q}_{T_{\b^{(n)}} , F} \right)_* (\mu \times m) \in \prob \left( \{0,1\}^F \right)$, which we now call $\nu_n$ for short.  The assumption that $\cov(\bar{Q}, T_{\b^{(n)}}, F, \mu \times m) \leq L$ for all $n$ says that for each $n$, there is a collection of $L$ words $w_1^{(n)}, \dots, w_L^{(n)} \in \{0,1\}^F$ such that the Hamming balls of radius $\epsilon$ centered at these words cover a set of $\nu_n$-measure at least $1-
\epsilon$.  Since $\{0,1\}^F$ is a finite set, there are only finitely many possibilities for the collection $( w_1^{(n)}, \dots, w_L^{(n)} )$.  Therefore by passing to a subsequence and relabeling we may assume that there is a fixed collection of words $w_1, \dots, w_L$ with the property that if we let $B_i$ be the Hamming ball of radius $\epsilon$ centered at $w_i$, then $\nu_n \left( \bigcup_{i=1}^L B_i \right) \geq 1-
\epsilon$ for every $n$.

Now, by Lemma \ref{lem: ConvergingCocycles}, the map $\bar{Q}_{T_{\b^{(n)}} , F}$ agrees with $\bar{Q}_{T_\a , F}$ on a set of measure converging to $1$ as $n \to \infty$.  This implies that the measures $\nu_n$ converge in the total variation norm on $\prob \left( \{0,1\}^F \right)$ to $\nu := \left( \bar{Q}_{T_\a , F} \right)_*(\mu \times m)$.  Since $\nu_n \left( \bigcup_{i=1}^L B_i \right) \geq 1-
\epsilon$ for every $n$, we conclude that $\nu \left( \bigcup_{i=1}^L B_i \right) \geq 1-
\epsilon$ also, which implies that 
$\cov(\bar{Q}, T_\a, F, \mu \times m) \leq L$ as desired.
\end{proof}

Define $\mathcal{U}_N := \{\a \in \operatorname{Co}(X) : 
\cov(\bar{Q}, T_\a, F_n, \mu \times m) > 2a_n \text{ for some $n > N$} \}$.  By Lemma \ref{lem: OpenSetOfCocycles}, each $\mathcal{U}_N$ is a union of open sets and therefore open.  Also, $\bigcap_{N} \mathcal{U}_N$ is exactly the set of $\a \in \operatorname{Co}(X)$ satisfying 
$\cov(\bar{Q}, T_\a, F_n, \mu \times m) > 2a_n$ infinitely often.  Therefore, by the Baire category theorem, in order to prove Theorem \ref{thm: TechnicalMainResult} it suffices to prove

\begin{prop} \label{prop: FinitaryVersionOfMainResult}
For each $N$, $\mathcal{U}_N$ is dense in $\operatorname{Co}(X)$.
\end{prop}

\noindent The proof of this proposition is the content of the next section.

\section{Proof of Proposition \ref{prop: FinitaryVersionOfMainResult}}  \label{sec: ConstructionSection}

\subsection{Setup}
Let $N$ be fixed and let $\a_0$ be an arbitrary cocycle.  Consider an arbitrary neighborhood of $\a_0$ determined by a finite set $F \subseteq G$ and $\eta > 0$.  We can assume without loss of generality that $\eta \ll \epsilon = 1/100
$.  We will produce a new cocycle $\a \in \mathcal{U}_N$ such that there is a set $X'$ of measure $\geq 1-\eta$ on which $\a_f(x) = (\a_0)_f(x)$ for all $f \in F$.  The construction of such an $\a$ is based on the fact that the orbit equivalence relation $R$ is hyperfinite.

\begin{thm}[\!\!{\cite[Theorem 10]{connes1981amenable}}] \label{thm: Hyperfinite}
There is an increasing sequence of equivalence relations $R_n \subseteq X \times X$ such that
\begin{itemize}
    \item each $R_n$ is measurable as a subset of $X \times X$,
    \item every cell of every $R_n$ is finite, and
    \item $\bigcup R_n$ agrees $\mu$-a.e. with $R$.
\end{itemize}
\end{thm}

\noindent Fix such a sequence $(R_n)$ and for $x \in X$, write $R_n(x)$ to denote the cell of $R_n$ that contains $x$.

\begin{lem} \label{lem: CellsContainFOrbits}
There exists an $m_1$ such that $\mu \{ x \in X : T^F x \subseteq R_{m_1}(x) \} > 1-\eta$.
\end{lem}
\begin{proof}
Almost every $x$ satisfies $T^G x = \bigcup_{m} R_m(x)$, so in particular, for $\mu$-a.e. $x$, there is an $m_x$ such that $T^F x \subseteq R_m(x)$ for all $m \geq m_x$.  Letting $X_\ell = \{x \in X : m_x \leq \ell\}$, we see that the sets $X_\ell$ are increasing and exhaust almost all of $X$.  Therefore we can pick $m_1$ so that $\mu(X_{m_1}) > 1-\eta$.   
\end{proof}

Now we drop $R_1, \dots, R_{m_1 - 1}$ from the sequence and assume that $m_1 = 1$.

\begin{lem}
There exists a $K$ such that $\mu\{x : |R_1(x)| \leq K \} > 1-\eta$.
\end{lem}
\begin{proof}
Every $R_1$-cell is finite, so if we define $X_k = \{x \in X : |R_1(x)| \leq k \}$, then the $X_k$ are increasing and exhaust all of $X$.  So we pick $K$ so that $\mu(X_K) > 1-\eta$.
\end{proof}

Continue to use the notation $X_K = \{x \in X : |R_1(x)| \leq K \}$.

\begin{lem} \label{lem: FolnerSetContaintsMostlySmallCells}
For all $n$ sufficiently large, $\mu \{x \in X : \frac{|(T^{F_n} x) \cap X_K|}{|F_n|} > 1-2\eta \} > 1-\eta$.
\end{lem}
\begin{proof}
We have $|(T^{F_n} x) \cap X_K| = \sum_{f \in F_n} 1_{X_K}(T^f x)$.  By the mean ergodic theorem, we get
\[
\frac{|(T^{F_n} x) \cap X_K|}{|F_n|} \ \to \ \mu(X_K) \ > \ 1-\eta \qquad \text{in probability as $n \to \infty$.}
\]
Therefore, in particular, $\mu \left\{ x \in X : \frac{|(T^{F_n} x) \cap X_K|}{|F_n|} > 1-2\eta \right\} \to 1$ as $n \to \infty$, so this measure is $>1-\eta$ for all $n$ sufficiently large.
%
\end{proof}

From now on, let $n$ be a fixed number that is large enough so that the above lemma holds, $n > N$, and $\frac12 \exp \left( \frac{1}{8K^2} \cdot |F_n| \right) > 2a_n$.  This is possible because $(a_n)$ is assumed to be subexponential in $|F_n|$.  The relevance of the final condition will appear at the end.

\begin{lem} \label{lem: Level2CellsContainFolnerOrbits}
There is an $m_2$ such that $\mu \{x \in X : T^{F_n} x \subseteq R_{m_2}(x) \} > 1-\eta$.
\end{lem}
\begin{proof}
Same proof as Lemma \ref{lem: CellsContainFOrbits}.
\end{proof}

Again, drop $R_2, \dots, R_{m_2-1}$ from the sequence of partitions and assume $m_2 = 2$.

\subsection{Construction of the perturbed cocycle}

Let $(R_n)$ be the relabeled sequence of equivalence relations from the previous section.  The following measure theoretic fact is well known.  Recall that two partitions $P$ and $P'$ of $I$ are said to be {\bf independent} with respect to $m$ if $m(E \cap E') = m(E)m(E')$ for any $E \in P$, $E' \in P'$.

\begin{lem} \label{lem: MakePartitionsIndependent}
Let $P$ and $P'$ be two finite partitions of $I$.  Then there exists a $\varphi \in \aut(I,m)$ such that $P$ and $\varphi^{-1}P'$ are independent with respect to $m$.
\end{lem}

\begin{prop} \label{prop: PerturbedCocycle}
For any $\a_0 \in \operatorname{Co}(X)$, there is an $\a \in \operatorname{Co}(X)$ such that
\begin{enumerate}
    \item $\a_g(x) = (\a_0)_g(x)$ whenever $(x, T^g x) \in R_1$, and
    \item for $\mu$-a.e. $x$, the following holds.  If $C$ is an $R_1$-cell contained in $R_2(x)$, consider the map $Y_C: t \mapsto \bar{Q}_{T_\a, \{g : T^g x \in C \}}(x,t)$ as a random variable on the underlying space $(I,m)$. Then as $C$ ranges over all such $R_1$-cells, the random variables $Y_C$ are independent.
    %
\end{enumerate}
\end{prop}

\begin{proof}
We give here only an intuitive sketch of the proof and leave the full details to Appendix \ref{app: Measurability}.  It is more convenient to adopt the perspective of a cocycle as a map $\a: R \to \aut(I,m)$ satisfying the condition $\a(x,z) = \a(y,z) \circ \a(x,y)$.


{\bf Step 1.} For $(x,y) \in R_1$, let $\a(x,y) = \a_0(x,y)$. 

{\bf Step 2.}  Fix an $R_2$-cell $\bar{C}$.  Enumerate by $\{C_1, \dots, C_k\}$ all of the $R_1$-cells contained in $\bar{C}$ and choose from each a representative $x_i \in C_i$.  

{\bf Step 3.}  Let $\pi$ denote the partition $\{[0,1/2), [1/2, 1]\}$ of $I$.  Define $\a(x_1, x_2)$ to be an element of $\aut(I,m)$ such that 
\[
\bigvee_{y \in C_1} \a(x_1, y)^{-1} \pi \qquad \text{and} \qquad \a(x_1, x_2)^{-1} \left(\bigvee_{y \in C_2} \a(x_2, y)^{-1} \pi \right)
\]
are independent.  These expressions are well defined because $\a$ has already been defined on $R_1$ and we use Lemma \ref{lem: MakePartitionsIndependent} to guarantee that such an element of $\aut(I,m)$ exists.  

{\bf Step 4.}  There is now a unique way to extend the definition of $\a$ to $(C_1 \cup C_2) \times (C_1 \cup C_2)$ that is consistent with the cocycle conditition.  For arbitrary $y_1 \in C_1, y_2 \in C_2$, define
\begin{align*}
    \a(y_1, y_2) \ &= \ \a(x_2, y_2) \circ \a(x_1, x_2) \circ \a(y_1, x_1) \qquad \text{and} \\
    \a(y_2, y_1) \ &= \ \a(y_1, y_2)^{-1}.
\end{align*}
The middle term in the first equation was defined in the previous step and the outer two terms were defined in step 1.

{\bf Step 5.} Extend the definition of $\a$ to the rest of the $C_i$ inductively, making each cell independent of all the previous ones.  Suppose $\a$ has been defined on $\left( C_1 \cup \dots \cup C_j \right) \times \left( C_1 \cup \dots \cup C_j \right)$.  Using Lemma \ref{lem: MakePartitionsIndependent} again, define $\a(x_1, x_{j+1})$ to be an element of $\aut(I,m)$ such that
\[
\bigvee_{y \in C_1 \cup \dots \cup C_j} \a(x_1, y)^{-1}\pi \qquad \text{and} \qquad \a(x_1,x_{j+1})^{-1} \left( \bigvee_{y \in C_{j+1}} \a(x_{j+1}, y)^{-1}\pi  \right)
\]
are independent.  Then, just as in step 4, there is a unique way to extend the definition of $\a$ to all of $(C_1 \cup \dots \cup C_{j+1}) \times (C_1 \cup \dots \cup C_{j+1})$.  At the end of this process, $\a$ has been defined on all of $\bar{C} \times \bar{C}$.  This was done for an arbitrary $R_2$-cell $\bar{C}$, so now $\a$ is defined on $R_2$.  

{\bf Step 6.} For each $N \geq 2$, extend the definition of $\a$ from $R_N$ to $R_{N+1}$ with the same procedure, but there is no need to set up any independence.  Instead, every time there is a choice for how to define $\a$ between two of the cell representatives, just take it to be the identity.  This defines $\a$ on $\bigcup_{N \geq 1} R_N$, which is equal mod $\mu$ to the full orbit equivalence relation, so $\a$ is a well defined cocycle.

Now we verify the two claimed properties of $\a$.  Property (1) is immediate from step 1 of the construction.  To check property (2), fix $x$ and let $C_j$ be any of the $R_1$-cells contained in $R_2(x)$.  Note that the name $\bar{Q}_{T_\a, \{g : T^g x \in C_j\}}(x,t)$ records the data $\bar{Q}(T^g_\a(x,t)) = \bar{Q}(T^g x, \a_g(x) t) = \pi( \a_g(x) t)$ for all $g$ such that $T^g x \in C_j$, which, by switching to the other notation is the same data as $\pi(\a(x, y) t)$ for $y \in C_j$.  So, the set of $t$ for which $\bar{Q}_{T_\a, \{g : T^g x \in C_j\}}(x,t)$ is equal to a particular word is given by a corresponding particular cell of the partition $\bigvee_{y \in C_j} \a(x,y)^{-1} \pi = \a(x, x_1)^{-1} \left( \bigvee_{y \in C_j} \a(x_1, y)^{-1} \pi \right)$.  The construction of $\a$ was defined exactly so that the partitions $\bigvee_{y \in C_j} \a(x_1, y)^{-1} \pi$ are all independent and the names $\bar{Q}_{T_\a, \{g : T^g x \in _j\}}(x,t)$ are determined by these independent partitions pulled back by the fixed $m$-preserving map $\a(x, x_1)$, so they are also independent.

The reason this is only a sketch is because it is not clear that the construction described here can be done in a way so that the resulting $\a$ is a measurable function.  To do it properly requires a slightly different approach; see Appendix \ref{app: Measurability} for full details.
\end{proof}

Letting $\wt{X} = \{x \in X : T^F x \subseteq R_1(x) \}$, this construction guarantees that $\a_f(x) = (\a_0)_f(x)$ for all $f \in F, x \in \wt{X}$.  By Lemma \ref{lem: CellsContainFOrbits}, $\mu(\wt{X}) > 1-\eta$, so this shows that $\a$ is in the $(F,\eta)-$neighborhood of $\a_0$.

\subsection{Estimating the size of Hamming balls}

Let $\a$ be the cocycle constructed in the previous section.  We will estimate the $(\mu \times m)$-measure of $(\bar{Q}, T_\a, F_n)$-Hamming balls in order to get a lower bound for the covering number.  The following formulation of Hoeffding's inequality will be quite useful \cite[Theorem 2.2.6]{vershynin2018high}.

\begin{thm} \label{thm: Hoeffding}
Let $Y_1, \dots, Y_\ell$ be independent random variables such that each $Y_i \in [0,K]$ almost surely.  Let $a = \E{\sum Y_i}$.  Then for any $t>0$,
\[
\P\left(  \sum_{i=1}^{\ell} Y_i < a - t \right) \ \leq \ \exp\left( - \frac{2t^2}{K^2 \ell} \right).
\]
\end{thm}

Let $X_0 = \left\{x \in X : \frac{|(T^{F_n} x) \cap X_K|}{|F_n|} > 1-2\eta \text{ and } T^{F_n} x \subseteq R_2(x) \right\}$.  By Lemmas \ref{lem: FolnerSetContaintsMostlySmallCells} and \ref{lem: Level2CellsContainFolnerOrbits}, $\mu(X_0) > 1 - 2\eta$.  Also write $\mu \times m = \int m_x \,d\mu(x)$, where $m_x = \delta_x \times m$.

\begin{prop} \label{prop: MeasureOfHammingBall}
For any $(x,t) \in X_0 \times I$, 
\beq 
m_x\left( \B(\bar{Q}, T_\a, F_n, (x,t),\epsilon) \right) \ \leq \ \exp\left( - \frac{1}{8K^2} \cdot |F_n| \right).
\eeq
\end{prop}
\begin{proof} Let $\mathcal{C}$ be the collection of $R_1$-cells $C$ that meet $T^{F_n} x$ and satisfy $|C| \leq K$.  For each $C \in \mathcal{C}$, let $F_C = \{f \in F_n : T^f x \in C \}$.  Define
\[
Y(t') \ = \ |F_n| \cdot d_{F_n} \left( \bar{Q}_{T_\a, F_n}(x,t), \bar{Q}_{T_\a, F_n}(x,t') \right) \ = \ \sum_{f \in F_n} 1_{\bar{Q}(T_\a^f(x,t)) \neq \bar{Q}(T_\a^f(x,t'))},
\]
and for each $C \in \mathcal{C}$, define
\[
Y_C(t') \ = \ \sum_{f \in F_C} 1_{\bar{Q}(T_\a^f(x,t)) \neq \bar{Q}(T_\a^f(x,t'))}.
\]
Then we have
\[
Y(t') 
\ \geq \ \sum_{C \in \mathcal{C}} Y_C(t'),
\]
so to get an upper bound for $m_x \left( \B(\bar{Q}, T_\a, F_n, (x,t),\epsilon) \right) = m \left\{ t' : Y(t') < \epsilon |F_n| \right\}$, it is sufficient to control $m \left\{ t' : \sum_{C \in \mathcal{C}} Y_C(t') < \epsilon |F_n| \right\}$.

View each $Y_C(t')$ as a random variable on the underlying probability space $(I,m)$.  Our construction of the cocycle $\a$ guarantees that the collection of names $\bar{Q}_{T_\a, F_C}(x,t')$ as $C$ ranges over all of the $R_1$-cells contained in $R_2(x)$ is an independent collection.  Therefore, in particular, the $Y_C$ for $C \in \mathcal{C}$ are independent (the assumption that $x \in X_0$ guarantees that all $C \in \mathcal{C}$ are contained in $R_2(x)$). 

We also have that each $Y_C \in [0,K]$ and the expectation of the sum is
\begin{align*}
a \ &:= \ \sum_{C \in \mathcal{C}} \int Y_C(t') \,dm(t') \ = \ \sum_{C \in \mathcal{C}} \sum_{f \in F_C} \int 1_{\bar{Q}(T_\a^f(x,t)) \neq \bar{Q}(T_\a^f(x,t'))} \,dm(t') \ = \ \sum_{C \in \mathcal{C}} \frac{1}{2} |F_C| \\
&= \ \frac12 \sum_{C \in \mathcal{C}}|C \cap (T^{F_n} x)| \ > \ \frac12 (1-2\eta)|F_n|,
\end{align*}
where the final inequality is true because $x \in X_0$.  So, we can apply Theorem \ref{thm: Hoeffding} with $t = a - \epsilon|F_n|$ to conclude
\begin{align*}
m \left\{ t' : \sum_{C \in \mathcal{C}} Y_C(t') < \epsilon|F_n| \right\} \ &\leq \ \exp\left(  \frac{-2t^2}{K^2 |\mathcal{C}|}  \right) \ \leq \ \exp \left(  \frac{-2(1/2 - \eta - \epsilon)^2 |F_n|^2}{K^2 |F_n|}  \right) \\
&\leq \  \exp\left( - \frac{1}{8K^2} \cdot |F_n| \right).
\end{align*}
The final inequality holds because $\epsilon = 1/100$ and $\eta \ll \epsilon$ is small enough so that $1/2 - \eta - \epsilon > 1/4$.
\end{proof}

\begin{cor}
Let $y \in X_0$. If $B$ is any $(\bar{Q}, T_\a, F_n)$-Hamming ball of radius $\epsilon$, then $m_y(B) \leq \exp\left( -\frac{1}{8K^2} \cdot |F_n| \right)$.
\end{cor}
\begin{proof}
If $B$ does not meet the fiber above $y$, then obviously $m_y(B) = 0$.  So assume $(y,s) \in B$ for some $s \in I$.  Then applying the triangle inequality in the space $(\{0,1\}^{F_n}, d_{F_n})$ shows that $B \subseteq \B(\bar{Q}, T_\a, F_n, (y,s), 2\epsilon)$.  Now apply Proposition \ref{prop: MeasureOfHammingBall} with $2\epsilon$ in place of $\epsilon$.  The proof goes through exactly the same and we get the same constant $1/8K^2$ in the final estimate because $\epsilon$ and $\eta$ are small enough so that $1/2-\eta - 2\epsilon$ is still $> 1/4$.
\end{proof}

\begin{cor}
We have $\cov(\bar{Q}, T_\a, F_n, \mu \times m) \geq \frac12 \exp(\frac{1}{8K^2} \cdot |F_n|)$.
\end{cor}
\begin{proof}
Let $\{ B_i \}_{i=1}^\ell$ be a collection of $(\bar{Q}, T_\a,  F_n)$-Hamming balls of radius $\epsilon$ such that 
\[
(\mu \times m) \left( \bigcup B_i \right) \ > \ 1-
\epsilon.
\]
Then 
\begin{align*}
    1-
    \epsilon \ &< \ (\mu \times m) \left( \bigcup B_i \right) \ = \ (\mu \times m) \left( \bigcup B_i \cap (X_0 \times I) \right) + (\mu \times m) \left( \bigcup B_i \cap (X_0^c \times I) \right) \\
    &< \ \sum_{i=1}^\ell (\mu \times m)(B_i \cap (X_0 \times I)) + (\mu \times m)(X_0^c \times I) \\
    &< \ \sum_{i=1}^\ell \int_{y\in X_0} m_y(B_i) \,d\mu(y) + 2\eta \\
    &< \ \ell \cdot \exp\left( -\frac{1}{8K^2} \cdot |F_n| \right) + 2\eta,
\end{align*}
implying that $\ell > (1-
\epsilon-2\eta) \exp\left( \frac{1}{8K^2} \cdot |F_n| \right) > (1/2) \exp\left( \frac{1}{8K^2} \cdot |F_n| \right)$.
\end{proof}

Our choice of $n$ at the beginning now guarantees that 
$\cov(\bar{Q}, T_\a, F_n, \mu \times m) \geq \frac12 \exp(\frac{1}{8K^2} \cdot |F_n|) > 2a_n$, showing that $\a \in \mathcal{U}_N$ as desired.  This completes the proof of Proposition \ref{prop: FinitaryVersionOfMainResult}.

\appendix

\section{Measurability of the perturbed cocycle} \label{app: Measurability}
In this section, we give a more careful proof of Proposition \ref{prop: PerturbedCocycle} that addresses the issue of measurability.  We will need to use at some point the following measurable selector theorem \cite[Proposition 433F]{fremlin2003}.

\begin{thm} \label{thm: MeasurableSelector}
Let $(\Omega_1, \mathcal{F}_1)$ and $(\Omega_2, \mathcal{F}_2)$ be standard Borel spaces.  Let $\P$ be a probability measure on $(\Omega_1, \mathcal{F}_1)$ and suppose that $f: \Omega_2 \to \Omega_1$ is measurable and surjective.  Then there exists a measurable selector $g: \Omega_1 \to \Omega_2$ which is defined $\P$-a.e. (meaning $g(\omega) \in f^{-1}(\omega)$ for $\P$-a.e. $\omega \in \Omega_1$).
\end{thm}

Given $x \in X$, there is a natural bijection between $T^G x$ and $G$ because $T$ is a free action.  We can also identify subsets -- if $E \subseteq T^G x$, then we will write $\wt{E} := \{g \in G : T^g x \in E \}$.  Note that this set depends on the ``base point'' $x$.  If $x$ and $y$ are two points in the same $G$-orbit, then the set $\wt{E}$ based at $x$ is a translate of the same set based at $y$.  It will always be clear from context what the intended base point is.

\begin{defn}
A {\bf pattern} in $G$ is a pair $(H,\mathscr{C})$, where $H$ is a finite subset of $G$ and $\mathscr{C}$ is a partition of $H$.  
\end{defn}

\begin{defn}
For $x \in X$, define $\pat_n(x)$ to be the pattern $(H,\sC)$, where $H = \wt{R_n(x)}$ and $\mathscr{C}$ is the partition of $H$ into the sets $\wt{C}$ where $C$ ranges over all of the $R_{n-1}$-cells contained in $R_n(x)$.  
\end{defn}

\begin{lem}
$\pat_n(x)$ is a measurable function of $x$.
\end{lem}
\begin{proof}
Because there are only countably many possible patterns, it is enough to fix a pattern $(H, \mathscr{C})$ and show that $\{x : \pat_n(x) = (H, \sC) \}$ is measurable.  Enumerate $\sC = \{C_1, \dots, C_k \}$.  Saying that $\pat_n(x) = (H,\sC)$ is the same as saying that $T^H x = R_n(x)$ and each $T^{C_i}x$ is a cell of $R_{n-1}$.  We can express the set of $x$ satisfying this as
\begin{align*}
&\left( \bigcap_{i=1}^{k} \bigcap_{g,h \in C_i} \{x : (T^g x, T^h x) \in R_{n-1} \} \quad \cap \quad \bigcap_{(g, h) \in G^2 \setminus \bigcup (C_i \times C_i)} \{x : (T^g x, T^h x) \not\in R_{n-1} \} \right) \quad \cap \\
&\left( \bigcap_{g \in H} \{x : (x, T^g x) \in R_n \} \quad \cap \quad \bigcap_{g \not\in H} \{x : (x,T^g x) \not\in R_n \}  \right).
\end{align*}
Because each $R_n$ is a measurable set and each $T^g$ is a measurable map, this whole thing is measurable.
\end{proof}

For each pattern $(H,\sC)$, let $X_{H,\sC}^{(n)} \ = \ \{x \in X : \pat_n(x) = (H, \sC) \}$.  We will define our cocycle $\a$ inductively on the equivalence relations $R_n$.  For each $n$, the sets $X_{H,\sC}^{(n)}$ partition $X$ into countably many measurable sets, so it will be enough to define $\a$ measurably on each $X_{H,\sC}^{(n)}$.  At this point, fix a pattern $(H,\sC)$, fix $n = 2$, and write $X_{H, \sC}$ instead of $X_{H,\sC}^{(2)}$.  Define
\begin{align*}
    \Omega_2^{H,\sC} \ &= \ \left\{ \psi : H \times H \to \aut(I,m) : \psi(h_1, h_3) = \psi(h_2, h_3) \circ \psi(h_1, h_2) \ \text{for all} \ h_1,h_2,h_3 \in H \right\}, \\
    \Omega_1^{H,\sC} \ &= \ \left\{ \sigma : \bigcup_{C \in \sC} C \times C \to \aut(I,m) : \sigma(g_1, g_3) = \sigma(g_2, g_3) \circ \sigma(g_1, g_2) \ \text{for all} \ g_1,g_2,g_3 \in G \right\}, \\
    \Omega_2^{H,\sC,\text{ind}} \ &= \ \left\{ \psi \in \Omega_2^{H,\sC} : \psi \ \text{is $(H,\sC)$-independent} \right\},
\end{align*}
where $\psi \in \Omega_2^{H,\sC}$ is said to be {\bf $(H,\sC)$-independent} if for any fixed $h_0 \in H$, the partitions
\[
\bigvee_{h \in C} \psi(h_0, h)^{-1} \pi
\]
as $C$ ranges over $\sC$ are independent with respect to $m$.

\begin{prop}
For every $\sigma \in \Omega_1^{H,\sC}$, there is some $\psi \in \Omega_2^{H,\sC,\text{ind}}$ that extends $\sigma$.
\end{prop}
\begin{proof}
The idea is exactly the same as the construction described in steps 3-5 in the sketched proof of Proposition \ref{prop: PerturbedCocycle}, but we write it out here also for completeness.

Enumerate $\sC = \{C_1, \dots, C_k \}$ and for each $i$ fix an element $g_i \in C_i$.  First, obviously we will define $\psi = \sigma$ on each $C_i \times C_i$.  Next, define $\psi(g_1, g_2)$ to be an element of $\aut(I,m)$ such that 
\[
\bigvee_{g \in C_1} \sigma(g_1, g)^{-1} \pi \qquad \text{and} \qquad \psi(g_1, g_2)^{-1} \left(\bigvee_{g \in C_2} \sigma(g_2, g)^{-1} \pi \right)
\]
are independent.  Then, define $\psi$ on all of $(C_1 \cup C_2) \times (C_1 \cup C_2)$ by setting
\begin{align*}
    \psi(h_1, h_2) \ &= \ \sigma(g_2, h_2) \circ \psi(g_1, g_2) \circ \sigma(h_1, g_1) \qquad \text{and} \\
    \psi(h_2, h_1) \ &= \ \psi(h_1, h_2)^{-1}
\end{align*}
for any $h_1 \in C_1, h_2 \in C_2$.  Continue this definition inductively, making each new step independent of all the steps that came before it.  If $\psi$ has been defined on $\left( C_1 \cup \dots \cup C_j \right) \times \left( C_1 \cup \dots \cup C_j \right)$, then define $\psi(g_1, g_{j+1})$ to be an element of $\aut(I,m)$ such that
\[
\bigvee_{g \in C_1 \cup \dots \cup C_j} \psi(g_1, g)^{-1}\pi \qquad \text{and} \qquad \psi(g_1,g_{j+1})^{-1} \left( \bigvee_{g' \in C_{j+1}} \sigma(g_{j+1}, g')^{-1}\pi  \right)
\]
are independent.  Then extend the definition of $\psi$ to all of $(C_1 \cup \dots \cup C_{j+1}) \times (C_1 \cup \dots \cup C_{j+1})$ in the exact same way.  

At the end of this process, $\psi$ has been defined on $(C_1 \cup \cdots \cup C_k) \times (C_1 \cup \cdots \cup C_k) = H \times H$, and it satisfies the cocycle condition by construction.  To verify that it also satisfies the independence condition, notice that the construction has guaranteed that
\[
\bigvee_{h \in C} \psi(g_1, h)^{-1}\pi
\]
are independent partitions as $C$ ranges over $\sC$.  To get the same conclusion for an arbitrary base point $h_0$, pull everything back by the fixed map $\psi(h_0, g_1)$.  Because this map is measure preserving, pulling back all of the partitions by it preserves their independence.
\end{proof}

Now we would like to take this information about cocycles defined on patterns and use it to produce cocycles defined on the actual space $X$.  Define the map $\sigma^{H,\sC}: X_{H,\sC} \to \Omega_1^{H,\sC}$ by $\sigma^{H,\sC}_x(g_1, g_2) := \a_0(T^{g_1}x, T^{g_2} x)$.  Note that this is a measurable map because $\a_0$ is a measurable cocycle.

By Theorem \ref{thm: MeasurableSelector} applied to the measure $\P = (\sigma^{H, \sC})_* (\mu (\cdot \,|\, X_{H,\sC})) \in \prob(\Omega_1^{H,\sC})$, we get a \emph{measurable} map $E^{H,\sC}: \Omega_1^{H,\sC} \to \Omega_2^{H,\sC,\text{ind}}$ defined $\P$-a.e. such that $E^{H,\sC}(\sigma)$ extends $\sigma$.  Denote the composition $E^{H,\sC} \circ \sigma^{H,\sC}$ by $\psi^{H, \sC}$ and write the image of $x$ under this map as $\psi^{H, \sC}_x$.  To summarize, for every pattern $(H,\sC)$, there is a measurable map $\psi^{H,\sC} : X_{H,\sC} \to \Omega_2^{H,\sC,\text{ind}}$ defined $\mu$-a.e. with the property that $\psi^{H, \sC}_x$ extends $\sigma^{H,\sC}_x$.

It is now natural to define our desired cocycle $\a$ on the equivalence relation $R_2$ by the formula $\a(x, T^g x) := \psi^{\pat_2(x)}_x(e, g)$.  It is then immediate to verify the two properties of $\a$ claimed in the statement of Proposition \ref{prop: PerturbedCocycle}.  The fact that $\a$ agrees with $\a_0$ on $R_1$ follows from the fact that $\psi_{H,\sC}$ extends $\sigma^{H,\sC}$ and the claimed independence property of $\a$ translates directly from the independence property that the $\psi^{H,\sC}_x$ were constructed to have (see also the discussion after step 6 in the sketched proof of Proposition \ref{prop: PerturbedCocycle}).  Also, $\a$ is measurable because for each fixed $g$, the map $x \mapsto \a(x, T^g x)$ is simply a composition of other maps already determined to be measurable.  The only problem is that $\a$, when defined in this way, need not satisfy the cocycle condition.  To see why, observe that the cocycle condition $\a(x, T^h x) = \a(T^g x, T^h x) \circ \a(x, T^g x)$ is equivalent to the condition
\beq \label{eq: ModifiedCocycleCondition}
\psi_x^{\pat_2(x)}(e,h) \ = \ \psi_{T^g x}^{\pat_2(T^g x)}(e, hg^{-1}) \circ \psi_x^{\pat_2(x)}(e,g).
\eeq
But in defining the maps $\psi^{H,\sC}$, we have simply applied Theorem \ref{thm: MeasurableSelector} arbitrarily to each pattern separately, so $\psi^{\pat_2(x)}$ and $\psi^{\pat_2(T^g x)}$ have nothing to do with each other.  However, we can fix this problem with a little extra work, and once we do, we will have defined $\a: R_2 \to \aut(I,m)$ with all of the desired properties.

Start by declaring two patterns equivalent if they are translates of each other, and fix a choice of one pattern from each equivalence class.  Since there are only countably many patterns in total, there is no need to worry about how to make this choice.  For each representative pattern $(H_0, \sC_0)$, apply Theorem \ref{thm: MeasurableSelector} arbitrarily to get a map $\psi^{H_0, \sC_0}$.  This does not cause any problems because two patterns that are not translates of each other can not appear in the same orbit (this follows from the easy fact that $\pat_2(T^g x) = g^{-1} \cdot \pat_2(x)$), so it doesn't matter that their $\psi$ maps are not coordinated with each other.  For convenience, let us denote the representative of the equivalence class of $\pat_2(x)$ by $\repp(x)$.  Now for every $x \in X$, let $g^*(x)$ be the unique element of $G$ with the property that $\pat_2(T^{g^*(x)} x) = \repp(x)$.  Notice that the maps $g^*$ and $\repp$ are both constant on each subset $X_{H,\sC}$ and are therefore measurable. 

Now for an arbitrary pattern $(H,\sC)$ and $x \in X_{H,\sC}$, we define the map $\psi^{H,\sC}$ by 
\[
\psi_x^{H,\sC}(g,h) \ := \ \psi_{T^{g^*(x)} x}^{\repp(x)} ( g\cdot  g^*(x)^{-1}, h\cdot g^*(x)^{-1} ).
\]
Notice that this is still just a composition of measurable functions, so $\psi^{H,\sC}$ is measurable.  All that remains is to verify that this definition satisfies \eqref{eq: ModifiedCocycleCondition}.  The right hand side of \eqref{eq: ModifiedCocycleCondition} is
\begin{align*}
&\psi_{T^{g^*(T^g x)} T^g x}^{\repp(T^gx)} ( e g^*(T^g x)^{-1}, hg^{-1} g^*(T^g x)^{-1} ) \ \circ \ \psi_{T^{g^*(x)} x}^{\repp(x)} ( e g^*(x)^{-1}, g g^*(x)^{-1} ) \\
= \ &\psi_{T^{g^*(x)g^{-1}}T^g x}^{\repp(x)} ( (g^*(x)g^{-1})^{-1}, hg^{-1} (g^*(x)g^{-1})^{-1} ) \ \circ \ \ \psi_{T^{g^*(x)} x}^{\repp(x)} (g^*(x)^{-1}, g g^*(x)^{-1} ) \\
= \ &\psi_{T^{g^*(x)} x}^{\repp(x)} ( g g^*(x)^{-1} , hg^*(x)^{-1} ) \ \circ \ \psi_{T^{g^*(x)} x}^{\repp(x)} (g^*(x)^{-1}, g g^*(x)^{-1} ) \\
= \ &\psi_{T^{g^*(x)} x}^{\repp(x)}( g^*(x)^{-1}, hg^*(x)^{-1} ), 
\end{align*}
which is by definition equal to the left hand side of \eqref{eq: ModifiedCocycleCondition} as desired.

This, together with the discussion surrounding \eqref{eq: ModifiedCocycleCondition}, shows that if we construct the maps $\psi^{H,\sC}$ in this way, then making the definition $\a(x, T^g x) = \psi_x^{\pat_2(x)}(e,g)$ gives us a true measurable cocycle with all of the desired properties.  Finally, to extend the definition of $\a$ to $R_n$ with $n \geq 3$, repeat the exact same process, except it is even easier because there is no need to force any independence.  The maps $\psi^{H,\sC}$ only need to be measurable selections into the space $\Omega_2^{H,\sC}$, and then everything else proceeds in exactly the same way.

\bibliography{mybib}{}
\bibliographystyle{halpha}

\end{document}